\documentclass[12pt,twoside]{amsart}
\usepackage{amssymb}
\usepackage{graphicx}
\usepackage[margin=1in]{geometry}

\newtheorem{theorem}{Theorem}[section]
\newtheorem{lemma}[theorem]{Lemma}

\newtheorem{conjecture}[theorem]{Conjecture}

\theoremstyle{definition}
\newtheorem{remark}[theorem]{Remark}
\newtheorem{example}[theorem]{Example}

\usepackage{latexsym}
\usepackage{color,hyperref}
\definecolor{darkblue}{rgb}{0,0,0.6}
\hypersetup{colorlinks,breaklinks,linkcolor=darkblue,urlcolor=darkblue,anchorcolor=darkblue,citecolor=darkblue}

\title[Warnaar's bijection and colored partition identities, I]{Warnaar's bijection and colored partition identities, I}

\author{Colin Sandon}
\address{Department of Mathematics\\ MIT\\ Cambridge, MA 02139-4307}
\email{csandon@mit.edu}
\author{Fabrizio Zanello} \address{Department of Mathematics\\ MIT\\ Cambridge, MA 02139-4307\\{\tiny and}}
\address{Department of Mathematical  Sciences\\ Michigan Tech\\ Houghton, MI  49931-1295}
\email{zanello@math.mit.edu}
\thanks{2010 {\em Mathematics Subject Classification.} Primary: 05A17; Secondary: 05A19, 11P83, 05A15.\\
{\em Key words and phrases.} Partition identity; Colored partition; Farkas-Kra identity; Identity of the Schr\"oter, Russell and Ramanujan type; Modular equation; Bijective proof; Warnaar's bijection; Euler's Pentagonal Number Theorem}

\begin{document}
\begin{abstract}
We provide a general and unified combinatorial framework for a  number of colored partition identities, which include the five,  recently proved analytically by B. Berndt, that correspond to the exceptional modular equations of prime degree due to H. Schr\"oter, R. Russell and S. Ramanujan. Our approach generalizes that of S. Kim, who has given a bijective proof for two of these five identities, namely the ones modulo 7 (also known as the Farkas-Kra identity) and modulo 3. As a consequence of our method, we determine bijective proofs also for the two highly nontrivial identities modulo 5 and 11, thus leaving open combinatorially only the one modulo 23.
\end{abstract}

\maketitle

\section{Introduction}

Colored partition identities are a very active   research area within the theory of integer partitions. In particular, they provide natural combinatorial interpretations for certain classes of objects coming from other mathematical fields, including  equations that involve modular forms or theta functions. The simplest and perhaps best known identity of this family is the so-called ``Farkas-Kra identity modulo 7'' (see \cite{FK}), which states that there are as many integer partitions of $2N+1$ into distinct odd parts as there are integer partitions of $2N$ into distinct even parts, provided the multiples of 7 appear in two different copies. A combinatorial proof of this result had been asked for by H.M. Farkas and I. Kra, R. Stanley, B. Berndt and a number of other authors, and was recently given by S. Kim \cite{Ki}.

The Farkas-Kra identity is part of a set of five exceptional colored partition identities, sometimes referred to as ``identities of the Schr\"oter, Russell and Ramanujan type'', which correspond to five,  conjecturally unique,  modular equations of prime degree, discovered independently  by H. Schr\"oter \cite{Sc}, R. Russell \cite{Ru1,Ru2} and S. Ramanujan \cite{BR,Ra}. These modular equations, respectively of degree 3, 5, 7, 11 and 23, as Berndt pointed out in \cite{Be},  appear to be the only ones of such a simple type. See \cite{Be} for an interesting and detailed discussion of the history of these equations. In fact, in  his paper, Berndt determined and proved analytically the five corresponding  partition identities. As Berndt remarked, however (see also M.D. Hirschhorn \cite{Hi}), these five identities remained ``manifestly mysterious'', as they still lacked ``simple  bijective proofs'', which ``would be of enormous interest''. 

Soon afterwards, S. Kim \cite{Ki}, who employed in a clever fashion a bijection of S.O. Warnaar \cite{Wa} and generalized one of his results, provided an entirely bijective proof of, among other facts, two of the above identities  --- the one modulo 7, as we have said, and that modulo 3. 

A main goal of this paper is to respond to Berndt's call for a unified combinatorial framework in which to look at the five identities of the Schr\"oter, Russell and Ramanujan type. In fact, extending Kim's idea, we prove  an equivalence between a very broad family of colored partition identities, which include the above five, and suitable equations in  $(\nu_1,\dots,\nu_{t};d_1,\dots,d_{t})$, where $t\geq 1$, the $\nu_i$ are partitions, and the $d_i$ are integers whose sum is odd.

In particular, our approach  allows us to prove bijectively two more of the identities of the Schr\"oter, Russell and Ramanujan type, namely those corresponding to the modular equations of degree 5 and 11, whose specific proofs turn out to be highly nontrivial. Unfortunately, we have not been able to show bijectively  the last identity, that modulo 23. We state its equivalent equation as Conjecture \ref{7.2}.

In a sequel to this paper \cite{CSFZ2}, we will prove, again as a consequence of our method, a number of new (and  challenging) colored partition identities.

\section{The master bijection}

Let us first briefly recall the main definitions from partition  theory that we are going to use in this paper. For an introduction,  a survey of the  main  techniques, or a discussion of the philosophy behind  this fascinating field, see e.g. \cite{And,AE,Pak}, Section I.1 of \cite{Ma}, and Section 1.8 of \cite{St0}. 

Given a nonnegative integer $N$, we say that the nonincreasing sequence $\lambda=(\lambda^{(1)},\dots, \lambda^{(s)})$ of nonnegative integers is a \emph{partition} of $N$, and often write $|\lambda|=N$, if $\sum_{i=1}^s\lambda^{(i)}=N$. The $\lambda^{(i)}$  are called the \emph{parts} of $\lambda $, and the number of parts of $\lambda $ is  its \emph{length}, denoted by $l(\lambda)$. As usual, we define $p(N)$ to be the number of partitions of $N$ into positive parts; thus $p(a)=0$ for $a<0$, and $p(0)=1$, since we adopt  the standard convention that $\emptyset $ is the only partition of $N=0$.

Finally, let $P$ be the set of all partitions into positive parts,  $D_0$  the set of partitions into distinct nonnegative parts, and  $D=P\cap D_0$ the set of partitions into distinct positive parts. For instance, $\lambda=(6,6,3)\in P$ has length $l(\lambda )=3$, and $\lambda=(7,6,3,0)\in D_0$ has length $l(\lambda )=4$.

We begin with the following crucial bijection due to S.O. Warnaar \cite{Wa}, who generalized an earlier bijection of E.M. Wright \cite{Wr}. As usual, we set $\binom{d}{2}=d(d-1)/2$, for any $d\in \mathbb Z$.

\begin{lemma}[\cite{Wa}]\label{warnaar}
There exists a bijection between the set of triples $(\alpha, \beta, d)$, where $\alpha \in D_0$, $\beta \in D$ and $d=l(\alpha )- l(\beta )$, and the set of pairs $(\nu, d)$, where $\nu \in P$ and $d\in \mathbb Z$, such that
$$|\alpha| + |\beta| = |\nu|+\binom{d}{2}.$$
\end{lemma}

\begin{proof}
See \cite{Wa}, pages 48--49, for a description of the bijection.
\end{proof}

The next  theorem is the main general result of this paper. (We present it in a form that suffices for our purposes, even though it could easily be stated in more general terms.) It is an immediate corollary of the following lemma:

\begin{lemma}\label{N}
Fix integers $t\geq 1$, $C_1,\dots,C_t\geq 1$, and $0\leq A_i\leq C_i/2$ for all $i=1,\dots ,t$. Let $S$ be the set containing one copy of all positive integers congruent to $\pm A_i$ modulo $C_i$ for each $i$, and let $D_S(N)$  be the number of partitions of $N$ into distinct elements of $S$, where we require such partitions to have an odd number of parts if no $A_i$ is equal to zero. Finally, set $r=|\{A_i=0\}|-1$, adopting the convention that $|X|=1$ if $X=\emptyset$. 

Then, for all $N\geq 1$, 
$$2^{r}\cdot D_S(N)$$
is the number of solutions $(\nu_1,\dots,\nu_{t};d_1,\dots,d_{t})$ to the equation
\begin{equation}\label{eee}
\sum_{i=1}^{t}C_i|\nu_i|+\sum_{i=1}^{t}C_i{{d_i}\choose{2}} +\sum_{i=1}^{t}A_{i}d_{i}=N,
\end{equation}
where $\nu_i\in P$ and $d_i\in \mathbb Z$ for all $i$, and $\sum_{i=1}^{t}d_i$ is odd.
\end{lemma}

\begin{proof} This proof will  greatly generalize, but proceed for the most part in a similar way to, Kim's combinatorial proof of the Farkas-Kra identity modulo 7 (cf. \cite{Ki}, second proof of Theorem 2.1). A substantial difference is that we are going  to push the bijectivity of this type of argument all the way through, so that Theorem \ref{main} below will give us (ii)  equivalent to (i), which is going to be the crucial tool in attacking the identities of the Schr\"oter, Russell and Ramanujan type. 

Fix $N\geq 1$. We start by assuming that all of the $A_i$ are positive, and consider any partition $\pi$ of $N$ into distinct elements of $S$. We first split $\pi$ into $t$ pairs of partitions $(\lambda_1, \mu_1), \dots, (\lambda_t, \mu_t)$, where, for any $i$, both $\lambda_i$ and $\mu_i$ are in $D$, all  parts of $\lambda_i$ come from the copy of the integers of $S$ that are congruent to $ A_i$ (mod $C_i$), and all  parts of $\mu_i$ come from the copy of the integers of $S$ that are congruent to $-A_i$ (mod $C_i$). 

Let us now construct a new partition $\pi^{\ast}$ from $\pi$, which we split as $(\lambda_1^{\ast}, \mu_1^{\ast}), \dots,(\lambda_t^{\ast}, \mu_t^{\ast})$, where (entrywise) $\lambda_i^{\ast}=(\lambda_i-A_i)/C_i$ and $\mu_i^{\ast}=(\mu_i+A_i)/C_i$, for all  $i$. Notice that, clearly, $\lambda_i^{\ast}\in D_0$ and $\mu_i^{\ast}\in D$, for all $i$. 

Set $$d_i=l(\lambda_i )-l(\mu_i )=l(\lambda_i^{\ast} )-l(\mu_i^{\ast}).$$ 

Note that $\sum_{i=1}^{t}d_i \equiv \sum_{i=1}^{t}(l(\lambda_i )+l(\mu_i ))=l(\pi)$ (mod $2$); that is, $\sum_{i=1}^{t}d_i$ is odd if and only if $\pi$ has an odd number of parts.

By Lemma \ref{warnaar}, the triples $(\lambda_i^{\ast}, \mu_i^{\ast},d_i)$ are in (Warnaar's) bijection with  pairs $(\nu_i, d_i)$, where $\nu_i \in P$ and $|\lambda_i^{\ast}| + |\mu_i^{\ast}| = |\nu_i|+\binom{d_i}{2}$.

Therefore, it is easy to see that
$$N=|\pi |=\sum_{i=1}^{t}(|\lambda_i |+|\mu_i |)=\sum_{i=1}^{t}C_i|\lambda_i^{\ast}|+\sum_{i=1}^{t}C_i|\mu_i^{\ast}|+\sum_{i=1}^{t}d_iA_i=\sum_{i=1}^{t}C_i|\nu_i|+\sum_{i=1}^{t}C_i\binom{d_i}{2}+\sum_{i=1}^{t}d_iA_i.$$

Since all previous steps are reversible, this implies that the number of solutions to  equation (\ref{eee}) is $D_S(N)=2^{0}\cdot D_S(N)$, as desired.

This completes the proof when all of the $A_i$ are positive.

Suppose now that some $A_i=0$. Let us assume, without loss of generality, that $A_1=A_2=\dots =A_{r+1}=0$ for some $r\geq 0$, and that all other $A_j\neq 0$. The proof of this case goes along the same lines, except that now  the partitions $\lambda_i^{\ast}$ are in $D$ (not in $D_0$), for all $i\leq r+1$. Therefore, it is easy to see that, for $i\leq r+1$, the same partition $\lambda_i$ corresponds to exactly two solutions to  equation (\ref{eee}) --- one given by operating with Warnaar's bijection with $\lambda^{\ast}_i$, and the other with $\lambda^{\ast}_i$ to which a 0 is added at the end. 

Thus, each of our partitions $\pi$ corresponds bijectively to $2^{r+1}$ solutions to  (\ref{eee}), when $\sum_{i=1}^{t}d_i$ is arbitrary. 

Now, it is immediate to see that, in every solution to (\ref{eee}), $d_i$ can be replaced by $1-d_i$, for any  $i\leq r+1$. Since the parity of $d_i$ and $1-d_i$ is  different, it follows that exactly $2^{r+1}/2=2^r$ of the solutions to  (\ref{eee}) corresponding to the partition $\pi$ yield an odd value for $\sum_{i=1}^{t}d_i$. This easily concludes the proof of the lemma.
\end{proof}

\begin{theorem}\label{main}
Consider the equation
\begin{equation}\label{t}
\sum_{i=1}^{t}C_i|\mu_i|+\sum_{i=1}^{t}C_i{{d_i}\choose{2}} +\sum_{i=1}^{t}A_{i}d_{i}=\sum_{i=1}^{t}C_i|\alpha_i|+\sum_{i=1}^{t}C_i{{e_i}\choose{2}} +\sum_{i=1}^{t}B_{i}e_{i}+m,
\end{equation}
for given integers $t\geq 1$, $C_1,\dots,C_t\geq 1$, $0\leq A_i\leq C_i/2$ and $0\leq B_i\leq C_i/2$ for all $i$, and $m\geq 0$. Let $S$ be the set containing one copy of all positive integers congruent to $\pm A_i$ modulo $C_i$ for each $i$, and $T$ the set containing one copy of all positive integers congruent to $\pm B_i$ modulo $C_i$ for each $i$. Let $D_S(N)$ (respectively, $D_T(N)$) be the number of partitions of $N$ into distinct elements of $S$ (respectively, $T$), where we require such partitions to have an odd number of parts if no $A_i$ (respectively, no $B_i$) is equal to zero. Finally, set $$p=|\{B_i=0\}|-|\{A_i=0\}|,$$
adopting the convention that $|X|=1$ if $X=\emptyset$. 

Then the following are equivalent:

\begin{enumerate}
\item[(i)] For any $N\geq N_0\ge 1$, the number of tuples  $(\mu_1,\dots,\mu_{t};d_1,\dots,d_{t})$ such that the left-hand side of (\ref{t}) equals $N$, $\mu_i\in P$ and $d_i\in \mathbb Z$ for all $i$, and $\sum_{i=1}^{t}d_i$ is odd, is equal to the number of tuples $(\alpha_1,\dots,\alpha_{t};e_1,\dots,e_{t})$ such that the right-hand side of (\ref{t}) equals $N$,  $\alpha_i\in P$ and $e_i\in \mathbb Z$ for all $i$, and $\sum_{i=1}^{t}e_i$ is odd;

\item[(ii)] For any $N\geq N_0\ge 1$, $$D_S(N)=2^p\cdot D_T(N-m).$$
\end{enumerate}
\end{theorem}

\begin{proof}
Straightforward from Lemma \ref{N}.
\end{proof}

\section{The identities of the Schr\"oter, Russell and Ramanujan type}

The object of the rest of this paper is to show bijectively, using Theorem \ref{main}, four  of the five partition identities proved by Berndt in \cite{Be}, which correspond to the five exceptional modular equations of prime degree due to Schr\"oter, Russell and Ramanujan, as we discussed in the introduction. They will be proved in Theorems \ref{5.222}, \ref{3.222}, \ref{4.222}, and \ref{6.222}. We have not been able  to show the identity modulo 23; we will state an equation equivalent to it via Theorem \ref{main} as Conjecture \ref{7.2}.

We start with the partition identities modulo 7 (i.e., the Farkas-Kra identity) and  modulo 3. These are the two of the five for which a bijective proof is already known, thanks to the work of Kim \cite{Ki} (we will just slightly modify Kim's bijection here so as to fit our setting).

\begin{lemma}\label{5.2}
Condition (i) of Theorem \ref{main} holds for $N_0=1$, $t=4$, $C_1=\dots=C_4=14$, $m=1$, and
$$(A_1,A_2,A_3,A_{4})=(1,3,5,7),{\ }{\ }(B_1,B_2,B_3,B_{4})=(0,2,4,6).$$
\end{lemma}

\begin{proof}
It is easy to verify that the result  follows by associating the tuple
$$(\mu_1,\mu_2,\mu_3,\mu_4;d_1=2s+1-k-l+n,d_2=k-n,d_3=l-n,d_4=n),$$
where $n,l,k,$ and $s$ are arbitrary integers, to the tuple
$$(\alpha_1=\mu_1,\alpha_2=\mu_2,\alpha_3=\mu_3,\alpha_4=\mu_4;e_1=2n+1-k-l+s,e_2=k-s,e_3=l-s,e_4=-s).$$
(This is exactly Kim's map of \cite{Ki}, Theorem 1.1, except that here we needed to set $e_4=-s$ in place of $e_4=s$.)
\end{proof}

\begin{example}\label{fk15}
For any $N\geq 1$, Lemma \ref{5.2} puts the tuples $(\mu_1,\dots,\mu_{4};d_1,\dots,d_{4})\in P^4\times\mathbb{Z}
^4$ such that the left-hand side of the following equation equals $N$ and $d_1+d_2+d_3+d_4$ is odd in bijection with the tuples $(\alpha_1,\dots,\alpha_{4};e_1,\dots,e_{4})\in P^4\times\mathbb{Z}^4$ such that the right-hand side equals $N$ and $e_1+e_2+e_3+e_4$ is odd:

\begin{equation}\label{fk}
14\sum_{i=1}^{4}|\mu_i|+14\sum_{i=1}^{4}{{d_i}\choose{2}} +1d_1+3d_2+5d_3+7d_4=14\sum_{i=1}^{4}|\alpha_i|+14\sum_{i=1}^{4}{{e_i}\choose{2}} +0e_1+2e_2+4e_3+6e_4+1.
\end{equation}

Let $N=15$. It is easy to check that there are exactly six such tuples. The left-hand side of (\ref{fk}) equals 15 for $(\mu_1,\mu_2,\mu_3,\mu_{4};d_1,d_2,d_3,d_{4})$ equal to:

$$((1),\emptyset,\emptyset,\emptyset;1,0,0,0),{\ }{\ }{\ }{\ }{\ }{\ }(\emptyset,(1),\emptyset,\emptyset;1,0,0,0),{\ }{\ }{\ }{\ }{\ }{\ }(\emptyset,\emptyset,(1),\emptyset;1,0,0,0),$$
$$(\emptyset,\emptyset,\emptyset,(1);1,0,0,0),{\ }{\ }{\ }{\ }{\ }{\ }(\emptyset,\emptyset,\emptyset,\emptyset;0,1,1,1),{\ }{\ }{\ }{\ }{\ }{\ }(\emptyset,\emptyset,\emptyset,\emptyset;0,1,1,-1).$$

The bijection given in the proof of Lemma \ref{5.2}  maps the above solutions, respectively, to the following six tuples $(\alpha_1,\alpha_2,\alpha_3,\alpha_{4};e_1,e_2,e_3,e_{4})$  for which the right-hand side of equation (\ref{fk}) equals 15:

$$((1),\emptyset,\emptyset,\emptyset;1,0,0,0),{\ }{\ }{\ }{\ }{\ }{\ }(\emptyset,(1),\emptyset,\emptyset;1,0,0,0),{\ }{\ }{\ }{\ }{\ }{\ }(\emptyset,\emptyset,(1),\emptyset;1,0,0,0),$$
$$(\emptyset,\emptyset,\emptyset,(1);1,0,0,0),{\ }{\ }{\ }{\ }{\ }{\ }(\emptyset,\emptyset,\emptyset,\emptyset;0,1,1,-1),{\ }{\ }{\ }{\ }{\ }{\ }(\emptyset,\emptyset,\emptyset,\emptyset;-1,0,0,0).$$
\end{example}

\begin{theorem}[\cite{Ki}]\label{5.222}
Let $S$ be the set containing one copy of the odd positive integers and one more copy of the odd positive multiples of 7, and $T$ the set containing one copy of the even positive integers and one more  copy of the even positive multiples of 7. Then, for any $N\geq 1$,
$$D_S(N)=D_T(N-1).$$
\end{theorem}

\begin{proof}
Straightforward from Theorem \ref{main} and Lemma \ref{5.2}.
\end{proof}

\begin{example}

Let $N=15$ in  Theorem \ref{5.222}. By Example \ref{fk15} and Lemma \ref{N}, we have $$D_S(15)=D_T(14)=6.$$

Indeed, it is easy to check that 15 can be partitioned in the following six ways into distinct odd positive integers, where the multiples of 7 appear in two copies, say $7n$ and $\overline{7n}$:
$$(15), (11,3,1), (9,5,1), (7,5,3),  (\overline{7}, 5,3), (7,\overline{7},1).$$

Similarly, 14 can be partitioned in the  following six ways into distinct even positive integers, where the multiples of 7 appear in two copies:
$$(14), (\overline{14}), (12,2), (10,4), (8,6), (8,4,2).$$
\end{example}

\begin{lemma}\label{3.2}
Condition (i) of Theorem \ref{main} holds for $N_0=1$, $t=4$, $C_1=\dots=C_4=6$, $m=1$, and
$$(A_1,A_2,A_3,A_{4})=(1,1,3,3),{\ }{\ } (B_1,B_2,B_3,B_{4})=(0,0,2,2).$$
\end{lemma}

\begin{proof}
The exact same bijection as in Lemma \ref{5.2} easily gives the result.
\end{proof}

\begin{theorem}[\cite{Ki}]\label{3.222}
Let $S$ be the set containing 2 copies of the odd positive integers and 2 more copies of the odd positive multiples of 3, and $T$ the set containing 2 copies of the even positive integers and 2 more copies of the even positive multiples of 3. Then, for any $N\geq 1$,
$$D_S(N)=2D_T(N-1).$$
\end{theorem}

\begin{proof}
Straightforward from Theorem \ref{main} and Lemma \ref{3.2}.
\end{proof}

Notice that the two equations we are going to deal with next, which  are equivalent, respectively, to the partition identities modulo 5  and 11 of the Schr\"oter, Russell and Ramanujan type, will be:

\small
\begin{align*}
&2|\mu_1|+2|\mu_2|+10|\mu_3|+10|\mu_4|+2{d_1 \choose 2}+2{d_2 \choose 2}+10{d_3 \choose 2}+10{d_4 \choose 2}+d_1+d_2+5d_3+5d_4=\\
&2|\alpha_1|+2|\alpha_2|+10|\alpha_3|+10|\alpha_4|+2{e_1 \choose 2}+2{e_2 \choose 2}+10{e_3 \choose 2}+10{e_4 \choose 2}+0e_1+0e_2+0e_3+0e_4+3
\end{align*}
\normalsize
and

\small
$$2|\mu_1|+22|\mu_2|+2{d_1 \choose 2}+22{d_2 \choose 2}+d_1+11d_2=2|\alpha_1|+22|\alpha_2|+2{e_1 \choose 2}+22{e_2 \choose 2}+0e_1+0e_2+3.$$
\normalsize

Therefore, one moment's thought gives that the type of argument that held for Lemmas \ref{5.2} and \ref{3.2}, where  the bijection between the solutions could simply be taken to be the identity on all the partitions $\mu_i$, will not apply here, where $m=3$. For instance, in the first of the two equations, the tuple
$$(\mu_1,\dots,\mu_4;d_1,\dots,d_4)=(\emptyset,\emptyset,\emptyset,(1);1,0,0,0),$$
which makes the left-hand side equal 11, must  be mapped to a tuple $(\alpha_1,\dots,\alpha_{4};e_1, \dots, e_{4})$ such that the right-hand side equals 11, so we clearly need to have the partition $\alpha_4=\emptyset$. An entirely similar argument  holds for the second equation. This is why the two corresponding  partition results will be far more difficult to treat bijectively than the previous ones.

The following lemma is a classical application of Euler's Pentagonal Number Theorem:

\begin{lemma}\label{lemma0}
For any $n>0$, 
$$\sum_{i\in \mathbb Z} (-1)^ip\left(n-\frac{i(3i-1)}{2}\right)=0.$$
\end{lemma}

\begin{proof}
See e.g. \cite{Pak}, formula 5.1.2, or \cite{St0}, equation (1.91).
\end{proof}

\begin{lemma}\label{lemma1}
Fix  arbitrary $C_1,\dots,C_t,A_1,\dots,A_t,B_1,\dots,B_t$, such that $0\le A_i\le {C_i/2}$ and $0\le B_i\le C_i/2$, for all $i=1,\dots,t$. Let $S_N$ be the set of all tuples of $t$ partitions and $t$ integers $(\mu_1,\dots,\mu_{t};d_1,\dots,d_{t})$ such that $\sum_{i=1}^{t} d_i$ is odd and

\[\sum_{i=1}^tC_i|\mu_i|+\sum_{i=1}^tC_i{{d_i}\choose{2}}+\sum_{i=1}^t A_id_i =N.\]
Similarly, let $T_N$ be the set of all tuples of $t$ partitions and $t$ integers $(\alpha_1,\dots,\alpha_{t};e_1,\dots,e_{t})$ such that $\sum_{i=1}^{t} e_i$ is odd and

\[\sum_{i=1}^{t}C_i|\alpha_i|+\sum_{i=1}^{t}C_i{{e_i}\choose{2}} +\sum_{i=1}^{t} B_ie_i +m=N,\] 
where $m$ is an integer chosen so that the smallest value of $N$ for which $T_N\ne \emptyset$ is also the second smallest value of $N$ for which $S_N\ne \emptyset$. Define $k$ to be the smallest value such that $S_k\ne \emptyset$.  Further, let $U_N$ be the union of the set of all tuples of $t$ integers $(d_1,\dots,d_{t})$ such that $\sum_{i=1}^{t} d_i$ is odd and

\[\sum_{i=1}^{t}C_i{{d_i}\choose{2}}+\sum_{i=1}^{t} A_id_i =N,\]
with $|S_k|$ copies of the set of all tuples of $t$ integers $(f_1,\dots,f_{t})$ such that $\sum_{i=1}^{t} f_i$ is odd and 

\[\sum_{i=1}^{t}C_i \frac{f_i(3f_i-1)}{2}+k=N.\] 
Finally, let $V_N$ be the union of the  set of all tuples of $t$ integers $(e_1,\dots,e_{t})$ such that $\sum_{i=1}^{t} e_i$ is odd and

\[\sum_{i=1}^{t}C_i{{e_i}\choose{2}}+\sum_{i=1}^{t} B_ie_i+m =N,\]
with $|S_k|$ copies of the set of all tuples of $t$ integers $(f_1,\dots,f_{t})$ such that $\sum_{i=1}^{t} f_i$ is even and 

\[\sum_{i=1}^{t} C_i\frac{f_i(3f_i-1)}{2}+k=N.\] 
Then $|S_N|=|T_N|$ for all $N>k$ if and only if $|U_N|=|V_N|$ for all $N$.
\end{lemma}

\begin{proof}
For every $N$, let $U_N^{\ast}$ be the set of all pairs consisting of a tuple of $t$ partitions $(\mu_1,\dots,\mu_{t})$ and an element of $U_{N-x}$, where $x=\sum_{i=1}^{t} C_i|\mu_i|$. Likewise,  let $V_N^{\ast}$ be the set of all pairs consisting of a tuple of $t$ partitions $(\alpha_1,\dots,\alpha_{t})$ and an element of $V_{N-x}$, where $x=\sum_{i=1}^{t} C_i|\alpha_i|$.

Obviously, if $|U_N|=|V_N|$ for all $N$, then $|U_N^{\ast}|=|V_N^{\ast}|$ for all $N$. Conversely, if $|U_N^{\ast}|=|V_N^{\ast}|$ and $|U_x|=|V_x|$ for all $x<N$, then all of the terms of $|U_N^{\ast}|$ and $|V_N^{\ast}|$ in which any of the partitions are nonempty cancel out, leaving $|U_N|=|V_N|$. 

Thus, $|U_N|=|V_N|$ for all $N$ if and only if $|U_N^{\ast}|=|V_N^{\ast}|$ for all $N$. 

So  we need only prove that  $|S_N|=|T_N|$ for all $N>k$ if and only if $|U_N^{\ast}|=|V_N^{\ast}|$ for all $N$. By definition, $|S_N|=|T_N|=0$ for all $N<k$, and $|T_k|=0$ as well. Hence, for any $N<k$, $|U_N^{\ast}|=|V_N^{\ast}|=0$, and it is easy to see that
$$|U_k^{\ast}|=|S_k|+0 =0+|S_k|\cdot 1=|V_k^{\ast}|.$$

Therefore, it suffices to show that $|U_N^{\ast}|-|S_N|=|V_N^{\ast}|-|T_N|$ for all $N>k$. This is equivalent to the existence of a bijection between the set of all $(\mu_1,\dots,\mu_{t};f_1,\dots,f_{t})$ such that $\sum_{i=1}^{t} f_i$ is odd and
$$\sum_{i=1}^{t} C_i|\mu_i|+\sum_{i=1}^{t} C_i f_i(3f_i-1)/2+k=N,$$
and the set of all $(\alpha_1,\dots,\alpha_{t};f_1,\dots,f_{t})$ such that $\sum_{i=1}^{t} f_i$ is even and
$$\sum_{i=1}^{t} C_i|\alpha_i|+\sum_{i=1}^{t} C_i f_i(3f_i-1)/2+k=N.$$

We can associate each element of either set with a tuple $(n_1,\dots,n_{t})$, where, for any index $i$,
$$n_i=|\mu_i|+f_i(3f_i-1)/2 \phantom{x} \text{ or } \phantom{x}  n_i=|\alpha_i|+{f_i(3f_i-1)/2},$$
as appropriate. For any given $(n_1,\dots,n_{t})$, it is easy to see that the difference between the number of elements of the second set associated with $(n_1,\dots,n_{t})$ and the number of elements of the first set associated with it, is
\[\prod_{i=1}^{t}\sum_{f_i\in \mathbb Z} (-1)^{f_i}p\left(n_i-\frac{f_i(3f_i-1)}{2}\right).\]

Thus, by Lemma \ref{lemma0},  unless $(n_1,\dots,n_{t})=(0,\dots,0)$, the  last displayed formula is $0$. But we have $ \sum_{i=1}^{t} C_i n_i=N-k$, which implies that $n_i=0$ for all $i$ if and only if $N=k$. This  proves the bijection between the two sets for any $N>k$, as desired.
\end{proof}

\begin{remark}
Notice that, given a bijection $f$ between $U_N$ and $V_N$, we can construct a bijection between $S_N$ and $T_N$ as follows. First, create a bijection $f^{\ast}$ between $U_N^{\ast}$ and $V_N^{\ast}$ by having $f^{\ast}$ leave their partition components unchanged and act as $f$ on their integer components. Also, let $g$ be a bijection between $U_N^{\ast}-S_N$ and $V_N^{\ast}-T_N$ (where these set differences are defined in the obvious way). Constructing a bijection between $S_N$ and $T_N$ is now a standard variation of the Garsia-Milne involution principle \cite{GM} (see e.g. \cite{St0}, formula (2.33)). For instance, construct a graph where every element of $U_N^{\ast}$ or $V_N^{\ast}$ is a vertex, and there is an edge between any pair of elements that are in correspondence through $f^{\ast}$ or $g$. Thus, each element of $S_N$ or $T_N$ has one edge in this graph, and each element of $U_N^{\ast}-S_N$ or $V_N^{\ast}-T_N$ has two, so every element of $S_N$ is the other endpoint of a path ending at an element of $T_N$, and vice-versa. Therefore, the desired bijection between $S_N$ and $T_N$ maps each element of $S_N$ or $T_N$ to the other end of its path.
\end{remark}

We are now ready for the bijective proofs of the two partition identities corresponding, respectively, to the modular equations of degree 5 and 11 of the Schr\"oter, Russell and Ramanujan type (see Theorems 4.2 and 6.2 of \cite{Be}).

\begin{lemma}\label{4.2}
Condition (i) of Theorem \ref{main} holds for $N_0=3$, $t=4$, $C_1=C_2=2$, $C_3=C_4=10$, $m=3$, and
$$(A_1,A_2,A_3,A_{4})=(1,1,5,5),{\ }{\ } (B_1,B_2,B_3,B_{4})=(0,0,0,0).$$
\end{lemma}

\begin{proof}
It is easy to check that, in the notation of Lemma \ref{lemma1}, we have $k=1$ and $|S_k|=4$. Thus, by Lemma \ref{lemma1}, proving the lemma is equivalent to showing the existence of a bijection from the union of the set of all quadruples $(d_1,d_2,d_3,d_4)$ with $d_1+d_2+d_3+d_4$ odd and the set containing $4$ copies of each quadruple $(f_1,f_2,f_3,f_4)$ with $f_1+f_2+f_3+f_4$ odd, to the union of the set of all quadruples $(e_1,e_2,e_3,e_4)$ with $e_1+e_2+e_3+e_4$ odd and the set containing $4$ copies of each quadruple $(f_1',f_2',f_3',f_4')$ with $f_1'+f_2'+f_3'+f_4'$ even, such that,  for every pair of corresponding elements, 

$$2{d_1 \choose 2}+2{d_2 \choose 2}+10{d_3 \choose 2}+10{d_4 \choose 2}+d_1+d_2+5d_3+5d_4$$
$$\text{ or }{\ }{\ }  f_1(3f_1-1)+f_2(3f_2-1)+5f_3(3f_3-1)+5f_4(3f_4-1)+1$$
$$=2{e_1 \choose 2}+2{e_2 \choose 2}+10{e_3 \choose 2}+10{e_4 \choose 2}+0e_1+0e_2+0e_3+0e_4+3$$
$$  \text{ or }{\ }{\ }  f_1'(3f_1'-1)+f_2'(3f_2'-1)+5f_3'(3f_3'-1)+5f_4'(3f_4'-1)+1.$$


Notice that, by replacing $d_i$ with $1/2-e_i$ in the above $d$-formula, for $i=1,\dots,4$, we obtain the $e$-formula.  Thus, if we apply the map $d_i=1/2-e_i$, we can view the $d$-tuples and $e$-tuples as all being in the set $D=\{d\in \mathbb{Z}^4\cup(\mathbb{Z}+1/2)^4:d_1+d_2+d_3+d_4\in 2\mathbb{Z}+1\}$. (A tuple $(d_1,d_2,d_3,d_{4})\in D$ will in some sense be considered ``of negative type'' if the $d_i$ are half-integers, since it will come from the opposite side of the bijection as the tuples in which the $d_i$ are integers.)

Furthermore, if we define a dot product so that $$(d_1,d_2,d_3,d_4)\cdot(d_1',d_2',d_3',d_4')=d_1 d_1'+d_2 d_2'+5d_3 d_3'+5d_4 d_4',$$
then every point in this set corresponds to a quadruple with a value in the previous equation of its ``length'' squared. 

Now, let 
\[V_{1}=(1,1,1,1), V_{2}=(1,-1,1,-1), V_{3}=(5,5,-1,-1), V_{4}=(5,-5,-1,1).\]

It is easy to check that these vectors are pairwise orthogonal. For each $i=1, \dots, 4$, set $M_i= \|V_i \|^2/12$. 

Thus, $M_1=1$, $M_2=1$, $M_3=5$, and $M_4=5$. Also, for arbitrary $d=(d_1,d_2,d_3,d_4)$, we clearly have that $d\cdot V_1$, $d\cdot V_2$, $d\cdot V_3$, and $d\cdot V_4$ are all odd integers, and $d\cdot V_3$ and $d\cdot V_4$ are divisible by $5$. It follows that $d\cdot V_i$ is an odd multiple of $M_i$, for all $d$ and $i$. 

Now, for each $d\in D$ and $i$, let
\[r_i(d)=d-\frac{d\cdot V_i}{6M_i}V_i.\]

It is easy to check  that $\|r_i(d)\|=\|d\|$, for all $d$ and $i$. If $d\cdot V_i\equiv 0\pmod{3M_i}$ then $\frac{d\cdot V_i}{6M_i}$ is a half-integer, so $r_i(d)$ is a vector that corresponds to an $e$-quadruple if $d$ corresponds to a $d$-quadruple, and vice-versa. Hence, we can map every point in $D$ that has a dot product with at least one of the $V_i$ that is divisible by $3M_i$, to a point of the opposite type and the same value in the above $d$-formula, by sending it to $r_i(d)$, where $i$ is the smallest integer such that $d\cdot V_i\equiv 0\pmod{3M_i}$. 

Note that $r_i(d)\cdot V_i=-d\cdot V_i$. Also, $r_i(d)$ has the same dot product with $V_j$ as $d$ does, for all $j\ne i$, because of the orthogonality of the vectors. This implies that $r_i(r_i(d))\cdot V_j=d\cdot V_j$ for all $j$, and thus that  $r_i(r_i(d))=d$. Therefore, the above map is an involution. 

Hence, now we only need to consider the points in $D$ whose dot products with $V_i$ are not divisible by $3M_i$, for any $i$. Let $d\in D$ be any such point. For each $i$, let $x_i$ be the nearest integer to $\frac{d\cdot V_i}{6M_i}$, $y_i=\frac{d\cdot V_i-6M_ix_i}{M_i}$, and $z=d-\sum_{i=1}^{4} \frac{x_i}{2}V_i$. 

For any $i$, $d\cdot V_i\equiv \pm M_i\pmod{6M_i}$, so $y_i=\pm 1$. Thus, by the Pythagorean Theorem we obtain:
\[\|d\|^2=\sum_{i=1}^{4} \frac{(d\cdot V_i)^2}{\|V_i\|^2}=\sum_{i=1}^{4} \frac{(6M_ix_i+M_iy_i)^2}{12M_i}=1+\sum_{i=1}^{4} M_ix_i(3x_i+y_i).\]

We easily have that $z$ must be either a quadruple of integers or a quadruple of half-integers, and $z\cdot V_i=y_i M_i=\pm M_i$, for each $i$. It is a simple exercise to verify that the only  $z$ that fit these criteria are: $(1,0,0,0)$, $(-1,0,0,0)$, $(0,1,0,0)$, and $(0,-1,0,0)$. 

Therefore, we can choose a bijection between the $4$ possible values of $z$ and the $4$ copies of each tuple $(f_1,\dots,f_4)$, and then map $d$ to the copy of $(f_1, \dots,f_4)=(-x_1 y_1,\dots,-x_{4} y_{4})$ corresponding to $z$. 

It easily follows that
\[f_1(3f_1-1)+f_2(3f_2-1)+5f_3(3f_3-1)+5f_4(3f_4-1)+1=\|d\|^2.\]

Also, the $y_i$ are determined by $z$, and for any given choice of $z$, since $x_i=-y_if_i$, the only $d$ that maps to a given tuple $(f_1, \dots,f_4)$ is $z-\sum_{i=1}^{4} \frac{y_i f_i}{2} V_i.$

Furthermore, the entries of such $d$ are all half-integers if $\sum_{i=1}^{4} f_i$ is odd, and integers if it is even. Thus, this map always takes elements of $D$ corresponding to tuples of $d$'s to tuples of $f$'s with an even sum, and elements of $D$ corresponding to tuples of $e$'s to tuples of $f$'s with an odd sum. This completes the proof of the lemma.
\end{proof}

\begin{theorem}\label{4.222}
Let $S$ be the set containing 4 copies of the odd positive integers and 4 more copies of the odd positive multiples of 5, and $T$ the set containing 4  copies of the even positive integers and 4 more copies of the even positive multiples of 5. Then, for any $N\geq 3$,
$$D_S(N)=8D_T(N-3).$$
\end{theorem}

\begin{proof}
Straightforward from Theorem \ref{main} and Lemma \ref{4.2}.
\end{proof}

\begin{lemma}\label{6.2}
Condition (i) of Theorem \ref{main} holds for $N_0=3$, $t=2$, $C_1=2$, $C_2=22$, $m=3$, and
$$(A_1,A_2)=(1,11),{\ }{\ } (B_1,B_2)=(0,0).$$
\end{lemma}

\begin{proof}
It is easy to check that, in the notation of Lemma \ref{lemma1}, we have $k=1$ and $|S_k|=2$. Thus, proving the lemma is equivalent to proving the existence of a bijection from the union of the set of all pairs $(d_1,d_2)$ with $d_1+d_2$ odd and the set containing $2$ copies of each pair $(f_1,f_2)$ with $f_1+f_2$ odd, to the union of the set of all pairs $(e_1,e_2)$ with $e_1+e_2$ odd and the set containing $2$ copies of each pair $(f_1',f_2')$ with $f_1'+f_2'$ even, such that, for every pair of corresponding elements, 
\begin{align*}
2{d_1 \choose 2}+22{d_2 \choose 2}+d_1+11d_2{\ }{\ }\text{ or }{\ }{\ }f_1(3f_1-1)+11f_2(3f_2-1)+1\\
=2{e_1 \choose 2}+22{e_2 \choose 2}+0e_1+0e_2+3{\ }{\ }\text{ or }{\ }{\ }f_1'(3f_1'-1)+11f_2'(3f_2'-1)+1.
\end{align*}




Thus, similarly to what we did in Lemma \ref{4.2}, we can apply the map $d_1=1/2-e_1$, $d_2=e_2-1/2$, in order to view the $d$-pairs and $e$-pairs as both being in the same set $D=\{d\in \mathbb{Z}^2\cup(\mathbb{Z}+1/2)^2:d_1+d_2\in 2\mathbb{Z}+1\}$. Notice that, for any $(d_1,d_2)\in D$, $d_1+d_2$ is odd. 

If $d_1+d_2\equiv 0\pmod{3}$, then the map
$$d_1'=d_1-\frac{11(d_1+d_2)}{6}, {\ }{\ }{\ }{\ } d_2'=d_2-\frac{d_1+d_2}{6}$$
yields a pair having the same value as $(d_1,d_2)$, since $(d_1')^2+11(d_2')^2=d_1^2+11d_2^2$. Furthermore, it is easy to see that $(d_1',d_2')$ is an $e$-tuple (i.e., it has half-integer entries) if and only if $(d_1,d_2)$ is a $d$-tuple (i.e., it has integer entries). Thus, this map cancels out all such elements. 

If $d_1+d_2\not\equiv 0\pmod{3}$, then let $x$ be the closest integer to $(d_1+d_2)/6$, and let $d_1'=d_1-11x/2$ and $d_2'=d_2-x/2$.

We have $d_1'+d_2'=\pm 1$, so there must exist an integer $y$ such that $d_1'=y/2\pm 1$ and $d_2'=-y/2$. This means that $d_1=y/2+11x/2\pm1$ and $d_2=-y/2+x/2$. It easily follows that in this case $(d_1,d_2)\in D$ has a value of
$$d_1^2+11d_2^2=11x(3x\pm1)+y(3y\pm1)+1,$$
and therefore we can map $(d_1,d_2)$ to a copy of $(f_1=\mp y,f_2=\mp x)$. 

Finally, it is a standard task to verify that $(d_1,d_2)$ and $(-d_1,-d_2)$ get mapped to the same pair $(f_1,f_2)$, and that, for any $(d_1,d_2)$, $f_1+f_2$ is even if $(d_1,d_2)$ is a $d$-tuple and odd if it is an $e$-tuple. This completes the bijection and the proof of the lemma.
\end{proof}

\begin{theorem}\label{6.222}
Let $S$ be the set containing 2 copies of the odd positive integers and 2 more copies of the odd multiples of 11, and $T$ the set containing 2  copies of the even positive integers and 2 more copies of the even multiples of 11. Then, for any $N\geq 3$,
$$D_S(N)=2D_T(N-3).$$\end{theorem}

\begin{proof}
Straightforward from Theorem \ref{main} and Lemma \ref{6.2}.
\end{proof}

Finally, we  state as a conjecture the ``missing lemma'' of this paper, whose bijective proof eludes us. By Theorem \ref{main}, such a proof will imply a bijective proof also for the last of the five identities of the Schr\"oter, Russell and Ramanujan type (the one modulo 23, proved analytically in \cite{Be}, Theorem 7.2), and will therefore complete our unified combinatorial approach to the five identities. 

\begin{conjecture}\label{7.2}
Condition (i) of Theorem \ref{main} holds for $N_0=3$, $t=12$, $C_1=\dots=C_{12}=46$, $m=3$, and
$$(A_1,\dots,A_{12})=(1,3,5,7,9,11,13,15,17,19,21,23),$$$$(B_1,\dots,B_{12})=(0,2,4,6,8,10,12,14,16,18,20,22).$$
\end{conjecture}

\noindent
\textbf{Corollary  to Conjecture \ref{7.2}.}
\emph{Let $S$ be the set containing one copy of the odd positive integers and one more copy of the odd positive multiples of 23, and $T$ the set containing one copy of the even positive integers and one more  copy of the even positive multiples of 23. Then, for any $N\geq 3$,
$$D_S(N)=D_T(N-3).$$}

\section*{Acknowledgements} This work, along with the subsequent paper \cite{CSFZ2}, is the result of the first author's MIT senior  thesis, done in Summer and Fall 2011 under the supervision of the second author, and funded by the Institute through two UROP grants. The second author warmly thanks Richard Stanley for his terrific hospitality  during the whole year, the MIT Math Department for partial financial support, and Dr. Gockenbach and the Michigan Tech Math Department, from which he was on partial leave, for extra Summer support. The two authors also wish to thank the anonymous referees for comments, and Abhinav Kumar, Joel Lewis, and Richard Stanley for helpful discussions related to the materials of this work. Finally, the second author wants to acknowledge to be quite a distant second: the first author has done the better part of this project.



\begin{thebibliography}{ll}
\bibitem{And} G. Andrews: ``The theory of Partitions'', Encyclopedia of Mathematics and its Applications, Vol. II, Addison-Wesley, Reading, Mass.-London-Amsterdam (1976).
\bibitem{AE} G. Andrews and K. Eriksson: ``Integer Partitions'', Cambridge University Press, Cambridge, U.K. (2004).
\bibitem{BR} B.C. Berndt: ``Ramanujan's Notebooks'', Part III, Springer-Verlag, New York (1991).
\bibitem{Be} B.C. Berndt: \emph{Partition-theoretic interpretations of certain modular equations of Schr\"oter, Russell, and Ramanujan}, Ann. Comb. \textbf{11} (2007), no. 2, 115--125.
\bibitem{FK} H.M. Farkas and I. Kra: \emph{Partitions and theta constant identities}, in: ``The Mathematics of Leon Ehrenpreis'', Contemp. Math., no. 251, Amer. Math. Soc., Providence, RI (2000), 197--203.
\bibitem{GM} A.M. Garsia and S.C. Milne: \emph{A Rogers-Ramanujan bijection}, J. Combin. Theory Ser. A \textbf{31} (1981), 289--339.
\bibitem{Hi} M.D. Hirschhorn: \emph{The case of the mysterious sevens}, Int. J. Number Theory \textbf{2} (2006), 213--216.
\bibitem{Ki} S. Kim: \emph{Bijective proofs of partition identities arising from modular equations}, J. Combin. Theory Ser. A \textbf{116} (2009), no. 3, 699--712.
\bibitem{Ma} I.G. Macdonald: ``Symmetric Functions and Hall Polynomials'', Second Ed., Oxford Mathematical Monographs, The Clarendon Press Oxford University Press (1995).
\bibitem{Pak} I. Pak: {\em Partition bijections, a survey},  Ramanujuan J. {\bf 12} (2006), 5--75.
\bibitem{Ra} S. Ramanujan: ``Notebooks'', Vol. 1-2, Tata Institute of Fundamental Research, Bombay, India (1957).
\bibitem{Ru1} R. Russell: \emph{On $\kappa \lambda-\kappa' \lambda'$ modular equations}, Proc. London Math. Soc. \textbf{19} (1887), 90--111.
\bibitem{Ru2} R. Russell: \emph{On modular equations}, Proc. London Math. Soc. \textbf{21} (1890), 351--395.
\bibitem{CSFZ2} C. Sandon and F. Zanello: \emph{Warnaar's bijection and colored partition identities, II}, preprint.
\bibitem{Sc} H. Schr\"oter: \emph{Beitr\"age zur Theorie der elliptischen Funktionen}, Acta Math. \textbf{5} (1884), 205--208.
\bibitem{St0} R. Stanley: ``Enumerative Combinatorics'', Vol. I, Second Ed., Cambridge University Press, Cambridge, U.K. (2012).
\bibitem{Wa} S.O. Warnaar: \emph{A generalization of the Farkas and Kra partition theorem for modulus 7}, J. Combin. Theory Ser. A \textbf{110} (2005), no. 1, 43--52.
\bibitem{Wr} E.M. Wright: \emph{An enumerative proof of an identity of Jacobi}, J. London Math. Soc. \textbf{40} (1965), 55--57. 
\end{thebibliography}
\end{document}